\documentclass[12pt]{amsart}

\usepackage{amsthm, amsmath, amssymb, amsrefs, fullpage}

\newcommand{\R}{\mathbb{R}}
\newcommand{\T}{\mathbb{T}} 
\newcommand{\C}{\mathbb{C}} 
\newcommand{\D}{\mathbb{D}} 
\newcommand{\Z}{\mathbb{Z}}
\newcommand{\cd}{\overline{\D}} 
\newcommand{\al}{\alpha}

\newcommand{\mcF}{\mathcal{F}}

\numberwithin{equation}{section}

\newtheorem{theorem}{Theorem}[section]
\newtheorem{prop}[theorem]{Proposition}
\newtheorem{conjecture}[theorem]{Conjecture}
\newtheorem{corollary}[theorem]{Corollary}
\newtheorem{lemma}[theorem]{Lemma}

\theoremstyle{definition}
\newtheorem{definition}[theorem]{Definition}

\newtheorem{example}[theorem]{Example}

\title{Extreme points and saturated polynomials}

\author{Greg Knese}

\address{Washington University in St. Louis\\ Department of Mathematics\\ St. Louis, MO 63130}
\email{geknese@wustl.edu}

\thanks{Partially supported by NSF grant DMS-1363239}

\date{\today}



\begin{document}

\begin{abstract}
We consider the problem
of characterizing the extreme points
of the set of analytic functions $f$ on
the bidisk with positive real part
and $f(0)=1$.  
If one restricts
to those $f$ whose Cayley transform
is a rational inner function, one
gets a more tractable problem.
We construct families of such $f$ 
that are extreme points and conjecture 
that these are all such extreme points.
These extreme points are constructed 
from polynomials dubbed $\T^2$-saturated, which 
roughly speaking means they have no zeros in the bidisk
and as many zeros as possible on the boundary without having
infinitely many zeros. 
\end{abstract}

\maketitle

\tableofcontents

\section{Introduction}

Let $PRP_d$ denote the set of holomorphic functions $f:\D^d \to RHP$ such
that $f(0) = 1$.  Here $\D^d$ is the $d$-dimensional unit polydisk in
$\C^d$ and $RHP$ is the right half plane in $\C$.  

The set $PRP_d$ is a normal family and a convex set. 
Rudin posed the following natural question in his 1970 ICM address \cite{RudinICM}.
\begin{quote}
What are the extreme
points of $PRP_d$?
\end{quote}

A complete answer is known only in the case $d=1$.
 The extreme points
of $PRP_1$ are the functions
\[
z \mapsto \frac{1+\bar{\al} z}{1-\bar{\al}z}
\]
where $\al \in \T$, the unit circle.  Remarkably, this fact
together with the Krein-Milman theorem suffices to prove the Herglotz representation
theorem for elements of $PRP_1$: for every $f\in PRP_1$ there is a unique probability
measure $\mu$ on $\T$ such that
\[
f(z) = \int_{\T} \frac{1+\bar{\zeta}z}{1-\bar{\zeta}z} d\mu(\zeta).
\]
In turn, this fact can be used to prove a variety of spectral theorems.

There is no known characterization of the extreme points of $PRP_d$
for $d>1$ and we would guess there is not a simple characterization.  
Every $f\in PRP_d$ has a Poisson type representation
\[
\text{Re} f(z) = \int_{\T^d} P_{z_1} P_{z_2}\cdots P_{z_d} d\mu
\]
where $P_{w}(\zeta) = \frac{1-|w|^2}{|1-\bar{w}\zeta|^2}$ is the Poisson kernel
and $\mu$ is a probability measure.  However, when $d>1$, $\mu$ cannot 
be an arbitrary probability measure;
it must satisfy additional moment conditions since $f$ is analytic.
Specifically, we must have
$\text{supp}(\hat{\mu}) \subset \Z_{+}^d \cup \Z_{-}^d$ where $\Z_+$ (resp. $\Z_{-}$) denotes
the set of non-negative (resp. non-positive) integers. 
One might expect that extreme points of $PRP_d$ would be represented via
measures with ``small'' support but McDonald has constructed an extreme point
whose boundary measure is absolutely continuous with respect to Lebesgue measure
on $\T^d$ (see \cite{McDonald2}).  Not all of the known results are negative in spirit however.

Forelli has a useful necessary condition for $f \in
PRP_d$ to be an extreme point \cite{Forelli}.  
Let $S_d$ denote
the Schur class of the polydisk; 
the set of holomorphic functions $F:\D^d \to \cd$.  
First, via a Cayley transform 
we may write
\begin{equation} \label{fF}
f = \frac{1+F}{1-F}
\end{equation}
where $F \in S_d$ and $F(0) = 0$.  Forelli has proven that if $F$ is
reducible in the sense that $F = GH$ with non-constant $G,H \in S_d$ with $G(0)=0$
(but with no assumption on $H(0)$), then $f$ is \emph{not} an extreme point of
$PRP_d$.  Forelli's approach actually gives an elementary proof of the
characterization of extreme points for the case $d=1$.

Since the general problem of characterizing extreme points
of $PRP_d$ seems difficult, it seems reasonable
to restrict the problem to a more tractable subclass of
functions.  We shall examine the class
\[
PRPR_d = \{ f \in PRP_d: F=\frac{f-1}{f+1} \text{ is a rational inner function} \}.
\]
\emph{Inner} means
 for almost every $z\in \T^d$, $|F(rz)|\to 1$ as $r\nearrow 1$.
This implies $\Re f(rz) \to 0$ at least away from $z\in \T^d$ such that 
$F(rz) \to 1$.  This class, while very special, is
also dense in $PRP_d$ using the topology of 
local uniform convergence.  (Rudin \cite{FTIP}
proves that rational inner functions are
dense in $S_d$ and it then follows that $PRPR_d$ is
dense in $PRP_d$.) Thus, we shall modify our
opening question from Rudin.

\begin{quote}
Which $f\in PRPR_d$ are extreme points of $PRP_d$?
\end{quote}

It is proven in Rudin \cite{FTIP} that rational inner functions are
of the form
\[
F(z) = \frac{\tilde{p}(z)}{p(z)}
\]
where $p \in \C[z_1,\dots,z_d]$ 
\begin{enumerate}
\item has no zeros in $\D^d$, 
\item multidegree at most $n=(n_1,\dots,n_d)$ 
(the degree in $z_j$ is at most $n_j$), and
\item the reflection of $p$ given by
\[
\tilde{p}(z) = {z_1}^{n_1} \cdots {z_d}^{n_d} \overline{p(1/\bar{z}_1, \dots, 1/\bar{z}_d)}
\]
has no factors in common with $p$.
\end{enumerate}

We say multidegree ``at most'' above to allow for $\tilde{p}$ to have a monomial factor.  
Since $p(0)\ne 0$, we get that $\tilde{p}$ has multidegree exactly $n$.
For this reason it often makes more sense to refer to the multidegree of $\tilde{p}$.  

\begin{definition}
We say $p\in \C[z_1,\dots,z_d]$ is \emph{scattering stable} if it satisfies conditions
(1)-(3) above for $n$ equal to the multidegree of $p$.
\end{definition}

 Some simple
examples of rational inner functions are
\begin{equation} \label{schurexamples}
G_1 = \frac{3z_1z_2-z_1-z_2}{3-z_1-z_2} \qquad G_2=\frac{2z_1z_2-z_1-z_1}{2-z_1-z_2}
\end{equation}
and corresponding elements of $PRP_2$ are
\begin{equation} \label{prpexamples}
g_1 = \frac{3-2(z_1+z_2) + 3z_1z_2}{3(1-z_1z_2)} \qquad g_2= \frac{(1-z_1)(1-z_2)}{1-z_1z_2}.
\end{equation}

One way to analyze a convex set is to examine its \emph{faces}.
Given an element $x$ of a convex set $S$, 
its face $\mcF(x)$ is the union of all
line segments $L$ in $S$ with $x$ in the interior of $L$.
We set $\mcF(x) = \{x\}$ when $x$ is an extreme point.
McDonald has characterized the faces of $f\in PRPR_d$ as
elements of the convex set $PRP_d$.

\begin{theorem}[McDonald \cite{McDonald1}] \label{McDonald}
Let $f\in PRPR_d$ which we write as 
$f = \frac{1+F}{1-F} = \frac{p+\tilde{p}}{p-\tilde{p}}$ where $F = \frac{\tilde{p}}{p}$ is 
a rational inner function, written as above with $\tilde{p}(0)=0$ and
multidegree $n$.  Then, 
every element of the face $\mcF(f)$ of $f$ in $PRP_d$ is
of the form
\[
\frac{p+\tilde{p}+2v}{p-\tilde{p}}
\]
where $v \in \C[z_1,\dots,z_d]$ has multidegree at most $n$,
satisfies $v=\tilde{v}$, $v(0)=0$, and $p+v$ and $p-a v$ have no zeros
in $\D^d$ for some $a>0$.  The face $\mcF(f)$ has real dimension at most
$\prod (n_j+1) - 2$.  
\end{theorem}

The condition with $a$ is just to make $f$
in the interior of an appropriate line segment.
This theorem does not solve our problem because the
condition that $p+v$ and $p-av$ have no zeros in $\D^d$ 
is difficult to work with. Nevertheless, the theorem
certainly provides a useful reduction.  Indeed,
it shows that $\mcF(f)$ does not depend
on the choice of underlying convex set ($PRP_d$ versus $PRPR_d$).

The original proof of McDonald involves one variable slices and the Herglotz representation theorem.
We shall give a very similar proof that uses only complex analysis (i.e. no measure theory).
The number $\prod(n_j+1) - 2$ simply counts the real-linear dimension 
of the set of $v$ of multidegree at most $n$ satisfying $v(0)=0$ and $v=\tilde{v}$.
In particular, if $p$ has no zeros in $\cd^d$, then
this dimension is exactly $\prod(n_j+1)-2$ 
since $p\pm tv$ will have no zeros in $\cd^d$ for $|t|$ small enough.
Thus, if $p$ has no zeros in $\cd^d$ then $f=(p+\tilde{p})/(p-\tilde{p})$
is never an extreme point unless $p$ depends on only one variable.

For example, the function $g_1$ in \eqref{prpexamples} is not extreme and we can perturb
$p_1=3-z_1-z_2$ using any $v=az_1 + \bar{a}z_2$ with small enough $a\in \C$.
In fact, $p_1+az_1+\bar{a} z_2$ has no zeros in $\D^2$ if and only if $|1-a|\leq 3/2$.

We have been able to prove the following add-on to McDonald's theorem, which
also says that the dimension of $\mcF(f)$ is maximal
if and only if $p$ has no zeros in $\cd^d$.

\begin{theorem}\label{addon}
Assume the setup of Theorem \ref{McDonald}.  
If $p$ vanishes at $\zeta \in \T^d$ to order $M$, then
$v$ vanishes at least to order $M$ at $\zeta$.
\end{theorem}

We can also say something about the homogeneous expansion
of $v$.  See Theorem \ref{vhomo} for details.

Unfortunately, Theorem \ref{addon} has limitations
because order of vanishing is a crude way
to count zeros in this setting.
Consider the example $g_2$ above.  
Since $p_2=2-z_1-z_2$ has a zero at $(1,1)$,
so must any valid perturbation $v=az_1+\bar{a}z_2$.
This implies $a=ir$ for some $r\in \R$.  
It turns out $g_2$ is actually 
an extreme point, but it takes additional
work to prove this (i.e. to show $r=0$).

Our main goal is to prove a theorem which describes a number of extreme points
of $PRP_d$ coming from $PRPR_d$
 in the case $d=2$.
The denominators of the $F$'s appearing in these extreme points have a special property.

\begin{definition}
Let $p\in \C[z_1,z_2]$ be scattering stable and of bidegree $(n_1,n_2)$. 
We say  $p$ is \emph{$\T^2$-saturated} if $p$ and $\tilde{p}$ have $2n_1 n_2$ common zeros on $\T^2$.  
\end{definition}

The common zeros of $p$ and $\tilde{p}$ are
counted with multiplicities as in B\'ezout's theorem. 
The number $2n_1n_2$ counts all of the common zeros
of $p$ and $\tilde{p}$ on $\C_{\infty}\times \C_{\infty}$.
See \cite{Fischer, Fulton} for more details about intersection
multiplicities.  

The class of $\T^2$-saturated polynomials is interesting in its own right.
In the paper \cite{GK}, we gave two related characterizations of $\T^2$-saturated polynomials.
One relates to a sums-of-squares formula for
scattering stable polynomials. 
If $p\in \C[z_1,z_2]$ 
is scattering stable of bidegree $(n_1,n_2)$ then
\begin{equation} \label{sos}
|p|^2-|\tilde{p}|^2 = (1-|z_1|^2)SOS_1 + (1-|z_2|^2)SOS_2
\end{equation}
where $SOS_j$ is sum of squared moduli of $n_j$ polynomials.
A scattering stable polynomial $p$ is $\T^2$-saturated if and only if the 
sums of squares terms $SOS_1,SOS_2$ are unique (see Corollary 13.6 of \cite{GK}).  

A second characterization says that
a scattering stable polynomial $p$ is $\T^2$-saturated if and only if
\[
\int_{\T^2} \left|\frac{f}{p}\right|^2 d\sigma =\infty 
\]
for all nonzero $f\in \C[z_1,z_2]$ satisfying $\deg_j f < \deg_j p$ for  $j=1,2$ (see
Corollary 6.5 of \cite{GK}).
Thus, $\T^2$-saturated polynomials have so many boundary zeros
that lower degree polynomials cannot match them (in the above sense).

A simple example of a $\T^2$-saturated polynomial is $p=2-z_1-z_2$.
Note $p$ and $\tilde{p}$ share a single zero $(1,1)$ on $\T^2$ but
it has multiplicity $2$. This can be computed using the resultant
of $p$ and $\tilde{p}$ with respect to $z_2$.

Our main theorem \emph{suggests} a third characterization of $\T^2$-saturated polynomials
and constructs a family of extreme points in $PRP_2$.

\begin{theorem} \label{thmsaturated} 
Let $p \in \C[z_1,z_2]$ be scattering stable and $\tilde{p}(0,0)=0$.
Let $f=\frac{p+\tilde{p}}{p-\tilde{p}}$.  
If $p$ is $\T^2$-saturated and $p-\tilde{p}$ is irreducible,
then $f$ is an extreme point of $PRP_2$.
\end{theorem}

For example, one can conclude the example $g_2$ is an extreme point
simply by noticing $p=2-z_1-z_2$ is $\T^2$-saturated 
and $p-\tilde{p} = 2(1-z_1z_2)$ is irreducible. 
On the other hand, $p$ being $\T^2$-saturated is not
sufficient for $f$ to be extreme as the example
\[
\frac{1-G_2}{1+G_2} = \frac{(1-z_1z_2)}{(1-z_1)(1-z_2)}
 = \frac{1}{2}\left(\frac{1+z_1}{1-z_1} + \frac{1+z_2}{1-z_2}\right)
\]
shows. 

Along the path to proving Theorem \ref{thmsaturated}
we have been able to prove a more refined version
of Theorem \ref{addon}.
Given two polynomials $p,q\in \C[z_1,z_2]$ with a common
zero $\zeta$, we let $I_{\zeta}(p,q)$ denote the intersection
multiplicity of $p,q$ at $\zeta$.  

\begin{theorem} \label{thmmults}
Let $f=\frac{p+\tilde{p}}{p-\tilde{p}} \in PRPR_2$ where
$p\in \C[z_1,z_2]$ is scattering stable.
Suppose 
\[
g = f+ 2\frac{v}{p-\tilde{p}}
\]
belongs to the face of $f$ as in Theorem \ref{McDonald}.
Suppose $p(\zeta)=0$ for some $\zeta\in \T^2$.  Then,
\[
I_{\zeta}(p,\tilde{p}) \leq I_{\zeta} (p+v,\tilde{p}+v).
\]
\end{theorem}

We have been unable to prove the following.

\begin{conjecture} \label{conjecture}
If $p\in \C[z_1,z_2]$ is scattering stable, $\tilde{p}(0,0)=0$, and $f=\frac{p+\tilde{p}}{p-\tilde{p}}$ is
an extreme point of $PRP_2$, then $p$ is $\T^2$-saturated and $p-\tilde{p}$ is irreducible.
\end{conjecture}

A very modest piece of evidence is the following.  See section \ref{modest}.

\begin{theorem}\label{oneone}
Conjecture \ref{conjecture} holds if $p$ has bidegree $(1,1)$.
\end{theorem}

The reason Theorem \ref{oneone} works out is that if $f\in PRPR_2$
is extreme and depends on both variables, then the associated $p$ must have at least one 
zero on $\T^2$.  In the case of degree $(1,1)$ polynomials, a single zero implies 
saturation.  In Section \ref{nonsaturated} we present an example of 
a non-saturated scattering stable polynomial with a zero on $\T^2$ and
we show that the associated $f$ is not extreme.  This at least
shows that $\T^2$-saturated should not be replaced with $p$
having at least one zero on $\T^2$.

One final piece of evidence in favor of
Conjecture \ref{conjecture} is the following.

\begin{theorem} \label{intthm}
Suppose $p\in \C[z_1,z_2]$ is scattering stable with $\tilde{p}(0,0)=0$.

If there exists $v\in\C[z_1,z_2]$ such that 
$\deg v \leq \deg \tilde{p}$, $v(0,0)=0$, $v=\tilde{v}$, and $p\pm v$ have
no zeros in $\D^2$, then $v/p \in L^2(\T^2)$.

Going in the other direction, if $p$ is not $\T^2$-saturated
then there exists nonzero $v\in \C[z_1,z_2]$ such that $v/p\in L^2(\T^2)$, 
$\deg v \leq \deg \tilde{p}$, $v=\tilde{v}$, and $v(0,0)=0$.
\end{theorem}

Of course the key thing missing in the last statement is 
$p\pm t v$ have no zeros in $\D^2$ for some small $t>0$.  This theorem
requires some machinery from \cite{GK}.

In Section \ref{secdets}
we discuss some connections to determinantal representations
for stable polynomials
and transfer function representations for analytic functions.
In particular, we point out that saturated polynomials $p$
possess symmetric, contractive determinantal representations,
and the corresponding $F=\tilde{p}/p$ possess symmetric unitary
transfer function realizations.  These notions are all defined in
Section \ref{secdets}.  

\section{Proof of McDonald's characterization of faces}

In this section we will give a slightly more elementary proof 
of McDonald's characterization of faces (Theorem \ref{McDonald}) 
for $f=\frac{1+F}{1-F}$ where
$F$ is rational inner.  It is essentially the same proof except we avoid
measure theory and gain a little more
information and the face of $f$.

We first tackle the one variable case where $F$ is a Blaschke product
which we may write as $F = \frac{\tilde{p}}{p}$ for a polynomial $p\in \C[z]$
having no zeros in $\cd$ (because zeros on $\T$ get cancelled out) 
and with $\tilde{p}(0)=0$.
Then, $f = \frac{p+\tilde{p}}{p-\tilde{p}}$.  

Let $f_{1}$ be in the face $\mcF(f)$ 
corresponding to $f$.  This means there exist $a>0$ and analytic $H$ 
such that $f_t= f+tH \in PRP_1$ for
$t \in [-a,1]$.  It will be convenient later to write $H = \frac{2v}{p-\tilde{p}}$
for some holomorphic function $v$.
Division by $p-\tilde{p}$ is no loss of generality 
since this function is non-vanishing in $\D$ (indeed, by the maximum principle $F\ne 1$ in $\D$).  
Note that
\[
f_{1} = \frac{p+\tilde{p}}{p-\tilde{p}} + \frac{2 v}{p-\tilde{p}}
=\frac{1+\frac{\tilde{p}+v}{p+v}}{1-\frac{\tilde{p}+v}{p+v}}.
\]

Let $A=\{a_1,\dots, a_n\}$ be the set of zeros of $p-\tilde{p}$ on $\T$.
Since $\Re f = \frac{a}{a+1} \Re f_{1} + \frac{1}{a+1} \Re f_{-a} \geq 0$ and 
since $\Re f = \frac{|p|^2-|\tilde{p}|^2}{|p-\tilde{p}|^2} = 0$ 
for $z \in \T\setminus A$ we see that for all $t\in [-a,1]$
\[
\Re f_{t}(z) \to 0
\]
as $z \to \T\setminus A$.

By the Schwarz reflection principle, the $f_{t}$ extend analytically
across $\T\setminus A$ via the formula $f_{t} (z) = -\overline{f_{t}(1/\bar{z})}$ 
which even extends analytically to $\infty$.  The singularities at
the $a_j$ are either removable or simple poles.  
(Without going into too much detail we can give an elementary explanation.
The derivative of $f_{t}$ cannot vanish
on $\T\setminus A$ because of local mapping properties.  Thus, $f_{t}(e^{i\theta})$
is imaginary-valued and monotone off singularities and this is enough to prove 
$f_{t}$ omits an imaginary line segment near singularities.  
This implies $f_{t}$ cannot have an essential singularity at any $a_j$
by either the Big Picard theorem or by conformal mapping and the Casorati-Weierstrass theorem.
Any poles on $\T$ must be simple because of local mapping properties of $1/f_{t}$.)
Therefore, the $f_{t}$ are rational and by their symmetry properties we get

\[
\begin{aligned}
f_1(z) &= \frac{p(z)+\tilde{p}(z)}{p(z)-\tilde{p}(z)} + \frac{v(z)}{p(z)-\tilde{p}(z)} \\
&= -\overline{f_1(1/\bar{z})} = - \frac{\overline{p(1/\bar{z})}+\overline{\tilde{p}(1/\bar{z})}}{
\overline{p(1/\bar{z})}-\overline{\tilde{p}(1/\bar{z})}} -\frac{\overline{v(1/\bar{z})}}{
\overline{p(1/\bar{z})}-\overline{\tilde{p}(1/\bar{z})}}
\end{aligned}
\]
and after some simplification we get
\[
\frac{v(z)}{p(z)-\tilde{p}(z)} = \frac{z^n \overline{v(1/\bar{z})}}{p(z)-\tilde{p}(z)}
\]
for $z\notin A$.  Necessarily, $v$ has no poles (because
simple poles in the above function are accounted for by the 
denominator) and hence $v$ must be a polynomial.
Also, $v(z) = z^n\overline{v(1/\bar{z})}$ and $v(0)=0$.  

Thus, every element of the face associated to $f$ is of the form
\[
\frac{p+\tilde{p}}{p-\tilde{p}} + \frac{2v}{p-\tilde{p}}
\]
where $v$ is a polynomial such that $v=\tilde{v}$, $v(0)=0$, $p+v$ has no
zeros in $\D$, and there exists $a>0$ such that $p-av$ has no zeros in $\D$.

Conversely, given $v$ satisfying all of the above conditions, 
$\frac{\tilde{p}+v}{p+v}, \frac{\tilde{p}-av}{p-av}$ are finite Blaschke products,
and $f=\frac{p+\tilde{p}}{p-\tilde{p}}$ is a convex combination of $f_{1},f_{-a}$ as
defined above.

This gives the desired characterization of faces in the one variable case.  

Next we look at several variables.  
Write $f=\frac{p+\tilde{p}}{p-\tilde{p}}\in PRPR_d$ where $p\in \C[z_1,\dots,z_d]$ is 
scattering stable and has multidegree at most $n=(n_1,\dots, n_d)$.

Now, if $f+tH \in PRP_d$ for $t \in [-a,1]$ and $H = \frac{2v}{p-\tilde{p}}$ for some
analytic $v:\D^d\to \C$,
then we claim $v$ is a polynomial of degree at most $n$ with $v=\tilde{v}$.

We examine slices $f_{\zeta}(z):= f(z \zeta)$ 
where $z\in \D, \zeta \in \T^d$. We will use this notation
for other functions, not just $f$. 

If we verify the hypotheses of the following lemma, then we are finished.
The lemma will be proven at the end.

\begin{lemma}\label{slices} Let $h:\D^d\to \C$ be analytic. 
If for each $\zeta \in \T^d$, $h_{\zeta}(z)$ is a 
(one variable) polynomial
of degree at most $|n|=n_1+\dots+n_d$ 
satisfying $h_{\zeta}(z) = \zeta^n z^{|n|} \overline{h_{\zeta}(1/\bar{z})}$, 
then
$h$ is a ($d$ variable) polynomial of multidegree 
at most $n$ and $h=\tilde{h}$.
\end{lemma} 

Fix $\zeta \in \T^d$.
Since $p$ has multidegree at most $n$, $p_{\zeta}$ has
degree at most $|n|$ and we can calculate that
\[
\zeta^n \widetilde{p_{\zeta}}(z) = \tilde{p}_{\zeta}(z).
\]
Then, for a choice of $\mu= \sqrt{\zeta^n}$
\[
f_{\zeta} = \frac{p_{\zeta} +\zeta^n \widetilde{p_{\zeta}}}{p_{\zeta} -\zeta^n \widetilde{p_{\zeta}}}
=\frac{\bar{\mu} p_{\zeta} +\mu \widetilde{p_{\zeta}}}{\bar{\mu}p_{\zeta} -\mu \widetilde{p_{\zeta}}}.
\]
The point is that we are matching the earlier formulation of the Cayley transform.
However, in this case it is possible for $p_{\zeta}$ and $\widetilde{p_{\zeta}}$
to have common zeros, necessarily on $\T$.  Let $q$ be a greatest common divisor
of $p_{\zeta}$ and $\widetilde{p_{\zeta}}$. We can arrange for $q=\tilde{q}$ by 
multiplying by an appropriate constant since
all of the roots of $q$ will be on $\T$.
Let $g=\bar{\mu} p_{\zeta}/q$ which has degree at most $|n|-\deg q$.

Then, $f_{\zeta} = \frac{g+\tilde{g}}{g-\tilde{g}}$ and $H_{\zeta} = \frac{2\bar{\mu}v_{\zeta}/q}{g-\tilde{g}}$.
We can now apply the one variable result
to see that 
$\bar{\mu} v_{\zeta}/q$ is a polynomial of degree
$|n|-\deg q$ and is equal to its reflection which means
\[
\bar{\mu} v_{\zeta}(z)/q(z) = \mu z^{|n|} \overline{v_{\zeta}(1/\bar{z})}/q(z)
\]
and this implies the hypotheses of Lemma \ref{slices}, proven below.  
This lemma can essentially by found in McDonald \cite{McDonald1} but
we include a proof for convenience.

\begin{proof}[Proof of Lemma \ref{slices}]
We write out the power series
$h(w) = \sum_{\alpha} a_{\alpha} w^{\alpha}$, so that
\[
\begin{aligned}
h_{\zeta}(z) &= \sum_{\alpha} a_{\alpha} z^{|\alpha|} \zeta^{\alpha} \\
&= \sum_{j\geq 0} z^j(\sum_{\alpha: |\alpha|=j}a_{\alpha} \zeta^\alpha) 
\end{aligned}
\]
and we see that for $j> |n|$, we have $0\equiv \sum_{|\alpha|=j}a_{\alpha} \zeta^{\alpha}$ 
and therefore $a_{\alpha} = 0$ for $|\alpha|>|n|$ since this is 
an identically zero trigonometric polynomial.

The reflection condition on $h_{\zeta}$ implies for $0\leq j \leq |n|$
\[
\sum_{\alpha: |\alpha|=j}a_{\alpha} \zeta^\alpha =
\sum_{\alpha:|\alpha|=|n|-j} \overline{a_{\alpha}} \zeta^{n-\alpha}. 
\]
By matching Fourier coefficients, we see that $a_{\alpha}=0$ if for some $k$ we have $\alpha_k >n_k$.
Thus, $h$ has multidegree at most $n$ and the sum on the right can be reindexed as
\[
\sum_{\alpha: |\alpha| = j} \overline{a_{n-\alpha}} \zeta^{\alpha}
\]
(where $a_{\beta}=0$ if undefined) so that $a_{\alpha} = \overline{a_{n-\alpha}}$.
This exactly means $h=\tilde{h}$ if we examine coefficients.
\end{proof}

This proves McDonald's characterization of faces with the additional information
that $v_{\zeta}$ vanishes to at least the same order at a point of $\T$ as $p_{\zeta}$.
Using work of \cite{GK} it is possible to show
$v$ vanishes to at least the same order
as $p$ at a point of $\T^d$.  

\begin{proof}[Proof of Theorem \ref{addon}]

We may assume $\zeta = \mathbf{1}:=(1,\dots,1)$.  
Theorem 14.1 of \cite{GK} says that if $p\in \C[z_1,\dots, z_d]$
has no zeros in $\D^d$ and $p$ vanishes at $\mathbf{1}$ to order $M$, meaning
\[
p(\mathbf{1} - \zeta) = \sum_{j=M}^{N} P_M(\zeta)
\]
where the $P_j$ are homogeneous of total degree $j$ and $P_M\ne 0$, then
$P_M$ has no zeros in $RHP^d$.
Then, the order of vanishing of $p_{\mathbf{1}}(z) = p(z\mathbf{1})$ at $z=1$
equals $M$ because $P_M((1-z)\mathbf{1}) = (1-z)^M P_M(\mathbf{1})$ and $P_M(\mathbf{1}) \ne 0$.
(Geometrically, the variety $p=0$ has no radial tangents on $\T^d$.)

Note both $p$ and $p+v$ have no zeros in $\D^d$.
Thus, if $p$ vanishes to order $M$ at $\mathbf{1}$, then $p_{\mathbf{1}}$ 
vanishes to order $M$ at $1$.   In turn, $v_{\mathbf{1}}$ vanishes to
at least order $M$, and then so does $(p+v)_{\mathbf{1}}$, and we then
conclude 
the order of vanishing of $p+v$ at $\mathbf{1}$ is at least $M$. 
\end{proof}

If we use an additional result from \cite{GK} 
we can say something about the bottom homogeneous term
of $v$ at $\mathbf{1}$.  

Proposition 14.5 of \cite{GK} says that if $p \in \C[z_1,\dots,z_d]$ 
has no zeros in $\D^d$ and vanishes to order $M$ at $\mathbf{1}$
then writing homogeneous expansions:
\[
p(\mathbf{1} - \zeta) = \sum_{j=M}^N P_j(\zeta) \qquad
\tilde{p}(\mathbf{1} - \zeta) = \sum_{j=M}^{N} Q_j(\zeta) 
\]
we have that $Q_M$ is a unimodular multiple of $P_M$, say $Q_M = \mu^2 P_M$ for
some $\mu \in \T$.  (It will be convenient to take a square root of $\mu^2$ later.)
Set $F=\tilde{p}/p$.  
It is not hard to prove $F(\mathbf{1}-\zeta) \to \mu^2$ as $\zeta \to 0$ non-tangentially in 
$RHP^d$.  The details are in Proposition 14.3 of \cite{GK}.

We also need to write out the homogeneous expansion of $h$
\[
v(\mathbf{1}-\zeta) = \sum_{j=M}^N V_j(\zeta).
\]
By the above work, $v$ vanishes at least to order $M$ at $\mathbf{1}$.
It is possible that $v$ vanishes to higher order, so we allow for $V_M=0$.

We assume for simplicity that $p\pm v$ have no zeros in $\D^d$; this
is true after rescaling $v$.  The following inequalities hold
\[
|p\pm v|^2 - |\tilde{p} \pm v|^2 \geq 0 \text{ on } \D^d
\]
(see Lemma 14.4 of \cite{GK}).
We can rewrite this as 
\begin{equation} \label{rewrite}
|p(z)|^2-|\tilde{p}(z)|^2 \geq 2|\text{Re}\left((\overline{p(z)}-\overline{\tilde{p}(z)})v(z)\right)|.
\end{equation}
Let $\zeta \in RHP^d$ and set $z=z(t)=\mathbf{1} - t\zeta$. 
 For $t>0$ small enough $z(t) \in \D^d$ and since
$P_M$ is a unimodular multiple of $Q_M$, $|p(z(t))|^2-|\tilde{p}(z(t))|^2$ vanishes
to order at least $M(M+1)$ in $t$.  This must also hold
for the right hand side of \eqref{rewrite} which implies
\[
\text{Re}\left( (1-\bar{\mu}^2)\overline{P_M(\zeta)} V_M(\zeta)\right) = 0
\]
which implies $(1-\bar{\mu}^2)\frac{V_M}{P_M} = ic$ for some $c \in \R$.

If $\mu^2\ne 1$, then
the constant $\frac{ic}{1-\bar{\mu}^2}$ is of the form $r\mu$ for 
$r\in \R$.  We summarize the above.

\begin{theorem}\label{vhomo}
 Assume the setup of Theorem \ref{McDonald}.  
Suppose $p$ vanishes at $\mathbf{1}$ to order $M$.
Let $\mu^2 = F(\mathbf{1})$ in the sense of a
non-tangential limit.  Assume $\mu^2\ne 1$.
Then the 
bottom homogeneous terms of $p(\mathbf{1}-\zeta), v(\mathbf{1}-\zeta)$ at
$\zeta=0$, say $P_M, V_M$, satisfy
\[
V_M = r \mu P_M
\]
for some $r \in \R$.
\end{theorem} 

This result is strong enough
to prove $g_2$ from the introduction
is extreme.
Note that for $p(z_1,z_2) = 2-z_1-z_2$ we have
\[
p(1-\zeta_1,1-\zeta_2) = \zeta_1+\zeta_2
\]
\[
\tilde{p}(1-\zeta_1,1-\zeta_2) =-\zeta_1-\zeta_2+2\zeta_1\zeta_2
\]
which means $\mu^2=-1$.  Then, any perturbation of $p$ must be
of the form $v(z_1,z_2)=ir z_1-irz_2$ where $r\in \R$ and 
\[
v(1-\zeta_1,1-\zeta_2) = ir(\zeta_2-\zeta_1)
\]
which is not a multiple of $\zeta_1+\zeta_2$ unless $r=0$.  

The theorem says nothing when $\mu^2= 1$
and this is for good reason.
Consider
\[
q = i(2-z_1-z_2)
\]
For $G = \frac{\tilde{q}}{q}$ we have $\lim_{r\nearrow 1} G(r,r) = 1$ and
$g=(1+G)/(1-G)$ is not extreme:
\[
g = \frac{1-z_1z_2}{(1-z_1)(1-z_2)} = \frac{1}{2}\left(\frac{1+z_1}{1-z_1} + \frac{1+z_2}{1-z_2}\right).
\]

\section{Preliminaries for Theorems \ref{thmsaturated} and \ref{thmmults}}\label{prelims}

McDonald's theorem reduces the possibilities for the face of 
a given $f$.  If $f$ is built out of a $\T^2$-saturated polynomial, 
we can limit $f$'s face further.

First, let us recall the Schur-Cohn test.

\begin{theorem}[Schur-Cohn test]
Let $p\in \C[z]$ and write $p(z)=\sum_{j=0}^{m}p_j z^j$.  Define
\[
P= \begin{pmatrix} p_0 & p_1 & \dots & p_{m-1} \\
 0 & p_0 & \dots & p_{m-2}\\
 \vdots & 0 & \ddots & \vdots \\
 0 & 0 & 0 & p_{0}
 \end{pmatrix} \quad 
 Q = \begin{pmatrix} \overline{p_{m}} & \overline{p_{m-1}} & \dots & \overline{p_1} \\
 0 & \overline{p_m} & \dots & \overline{p_2} \\
 \vdots & 0 & \ddots & \vdots \\
 0 & 0 & 0 & \overline{p_m}
 \end{pmatrix}
\]
Then, $p$ has no zeros in $\cd$ if and only if $P^*P-Q^*Q >0$.  
\end{theorem}

Actually this result is due to Schur while Cohn proved a 
generalization (see \cite{PtakYoung}).

We can apply this to bivariate polynomials $p\in \C[z,w]$ of 
bidegree $(n,m)$ as follows.  
Write $p(z,w) = \sum_{j=0}^{m} p_j(z)w^j$.  
Then, $\tilde{p}(z,w) = \sum_{j=0}^{m} \tilde{p}_{m-j}(z) w^j$
where $\tilde{p}_j(z) = z^{n} \overline{p_j(1/\bar{z})}$.  
Note that we use $(z,w)$ for the coordinates in $\C^2$ instead
of $(z_1,z_2)$.

Define
\begin{equation}\label{PQ}
P(z)= \begin{pmatrix} p_0(z) & p_1(z) & \dots & p_{m-1}(z) \\
 0 & p_0(z) & \dots & p_{m-2}(z)\\
 \vdots & 0 & \ddots & \vdots \\
 0 & 0 & 0 & p_{0}(z)
 \end{pmatrix} \quad 
 Q(z) = \begin{pmatrix} \tilde{p}_{m}(z) & \tilde{p}_{m-1}(z) & \dots & \tilde{p}_1(z) \\
 0 & \tilde{p}_m(z) & \dots & \tilde{p}_2(z) \\
 \vdots & 0 & \ddots & \vdots \\
 0 & 0 & 0 & \tilde{p}_m(z)
 \end{pmatrix}.
\end{equation}
Note that for $z\in \T$, $Q(z)$ is almost a direct analogue of $Q$ above; 
it differs by a factor of $z^n$.  This gets 
cancelled out in the Schur-Cohn matrix calculation.

Indeed, $p$ has no zeros in $\T\times \cd$ 
if and only if $T_p(z)= P(z)^*P(z)-Q(z)^*Q(z)>0$ for 
all $z\in \T$.  

The following formula holds for $z\in \T, w,\eta \in \C$
\begin{equation}\label{scformula}
\frac{\overline{p(z,\eta)}p(z,w)-\overline{\tilde{p}(z,\eta)}\tilde{p}(z,w)}{1-\bar{\eta} w} = (1,\bar{\eta},\dots,\bar{\eta}^{m-1})T_p(z) (1,w,\dots, w^{m-1})^t.
\end{equation}

\begin{lemma} \label{nozonface}
If $p$ has no zeros in $\T\times \D$ or no zeros in $\D^2$, then $P(z)^*P(z)-Q(z)^*Q(z)\geq 0$ for all $z\in \T$.
\end{lemma}
\begin{proof}
If $p$ has no zeros in $\T\times \D$, 
then the polynomial $p_r(z,w) = p(z,rw)$ has no zeros in $\T\times \cd$ for $r\in (0,1)$.
We can then form analogues of $P,Q$ corresponding to $p_r, \tilde{p_r}$ which we label $P_r, Q_r$.  The dependence on $r$ is evidently continuous.  Then, $P_r^*P_r-Q_r^*Q_r>0$ on $\T$. If we send $r\nearrow 1$, then $P^*P-Q^*Q\geq 0$ by continuity.  

If $p$ has no zeros in $\D^2$, then $p(rz,w)$ has no zeros in $\T\times \D$ for $r\in (0,1)$.  
We can again form analogues, say $P^r,Q^r$, of $P,Q$ depending on $r$ for which $(P^r)^*P^r-(Q^r)^*Q^r \geq 0$ on $\T$.
Then, send $r\nearrow 1$ to see $P^*P-Q^*Q \geq 0$ on $\T$.
\end{proof}

\begin{lemma} \label{scatnozonface}
If $p$ is scattering stable, then $p$ has no zeros in $(\T\times \D)\cup(\D\times\T)$.
\end{lemma}
\begin{proof}
For fixed $z_0\in \D$, $p(z_0,\cdot)$ has no zeros in $\D$.  By Hurwitz's theorem, if we send $z_0$ to $\T$, we see that $p(z_0,\cdot)$ either has no zeros in $\D$ or is identically zero.  In the latter case, $z-z_0$ divides $p$ and hence also $\tilde{p}$ since $z_0\in \T$.  Thus, $p$ is non-vanishing on $\T\times \D$ and by symmetry on $\D\times \T$.
\end{proof}

\begin{lemma}\label{possc}
Let $p\in \C[z,w]$ and form $P(z),Q(z)$ as above.
Then,
$P(z)^*P(z) - Q(z)^*Q(z) >0$ for all but finitely many $z\in \T$
if and only if
$p$ has no zeros in $(\T\times \cd)\setminus S$ where
$S$ consists of finitely many points in $\T^2$ and
finitely many ``vertical'' lines $z=z_0$ with $z_0\in \T$.

We can rule out vertical lines if we assume $p$ and $\tilde{p}$ 
have no common factors.  
\end{lemma}

We note that the condition $P^*P-Q^*Q>0$ for all but finitely many $z\in \T$ 
is equivalent to saying $P^*P-Q^*Q\geq 0$ and $\det(P^*P-Q^*Q)$ is not
identically zero.

\begin{proof}
By the Schur-Cohn test, $P(z)^*P(z)-Q(z)^*Q(z)>0$ 
for all but finitely many $z\in \T$ if and only if
 $w\mapsto p(z,w)$ has no zeros in $\cd$ for all but finitely many $z\in \T$.
The latter condition means there exists a finite set $S_1 \subset \T$ 
such that $w\mapsto p(z,w)$ has no zeros in $\cd$ except for when $z\in S_1$.  
By Hurwitz's theorem, if $z_0\in S_1$ then $w\mapsto p(z_0,w)$ 
either has no zeros in $\D$ or is identically zero. 
In the latter case, $z-z_0$ divides $p$ which means $p$ vanishes
on a vertical line.
In the former case, $p(z_0,\cdot)$ has no zeros in $\D$ (although it could have zeros in $\T$). 
 Thus, the condition $P^*P-Q^*Q>0$ for all 
but finitely many $z$ is equivalent to $p$ having
 no zeros in $(\T\times \cd) \setminus S$ where $S$ 
consists of finitely many vertical lines and finitely many points of $\T^2$.
\end{proof}

\begin{lemma} \label{commonfactor}
Assume $p\in\C[z,w]$ has bidegree $(n,m)$.
Let $r(z)$ be the resultant of $p$ and $\tilde{p}$ with respect
to the variable $w$.  Then, on $\T$
\[
r(z) = z^{nm} \det(P(z)^*P(z) - Q(z)^*Q(z))
\]
where $P,Q$ are defined in \eqref{PQ}.
In particular, 
the polynomials $p$ and $\tilde{p}$ have a common factor involving $w$
 if and only if $\det(P(z)^*P(z)-Q(z)^*Q(z))=0$ for all $z\in \T$.
\end{lemma}

\begin{proof}
This proof is from \cite{GW} (specifically Lemma 2.1.3 of that paper, 
which says their proof is inspired by similar
arguments for Bezoutians in \cite{LT}, Theorem 1 Section 13.3).  

Let $\tilde{P}(z) = z^n P(1/\bar{z})^*$ and $\tilde{Q}(z) = z^n Q(1/\bar{z})^*$.  Then,
\[
\begin{aligned}
R(z) &= \begin{pmatrix} P(z) & \tilde{Q}(z) \\ Q(z) & \tilde{P}(z) \end{pmatrix}\\
&= \begin{pmatrix} p_0(z) & \cdots & p_{m-1}(z) & p_{m}(z) & \ & \bigcirc \\
\ & \ddots & \vdots & \vdots & \ddots & \ \\
\bigcirc & \ & p_0(z) & p_1(z) & \cdots & p_m(z) \\
\tilde{p}_{m}(z) & \cdots & \tilde{p}_{1}(z) & \tilde{p}_{0}(z) & \ & \bigcirc \\
\ & \ddots & \vdots & \vdots & \ddots & \ \\
\bigcirc & \ & \tilde{p}_{m}(z) & \tilde{p}_{m-1}(z) & \cdots & \tilde{p}_0(z) \
\end{pmatrix} 
\end{aligned}
\]
is the resultant or Sylvester matrix of $p$ and $\tilde{p}$ with respect to $w$
and its determinant is the resultant $r$ of $p$ and $\tilde{p}$ with respect to $w$.
Since $P$ and $Q$ commute, we can compute
\[
r(z) = \det(\tilde{P}(z)P(z)-\tilde{Q}(z)Q(z))  
\]
and on $\T$ this agrees with $z^{nm}\det(P(z)^*P(z)-Q(z)^*Q(z))$.  
Since $\T$ is a set of uniqueness, this determinant is identically zero on $\T$ 
if and only if the resultant $r\equiv 0$ 
which by standard properties of the resultant holds
if and only if 
$p$ and $\tilde{p}$ have a common factor depending on $w$.
\end{proof}


We shall begin the proofs of Theorem \ref{thmsaturated} and
Theorem \ref{thmmults} simultaneously
and then diverge at a certain point.

Let $f=\frac{p+\tilde{p}}{p-\tilde{p}}$ be in $PRPR_2$ 
with associated scattering stable polynomial
$p\in \C[z,w]$. 

 By McDonald's characterization, 
points of the faces of
 $f$ are associated to 
 $v\in \C[z,w]$ satisfying: 
$v(0)=0$, $\deg v = \deg \tilde{p}$, $v=\tilde{v}$, and
there is an interval $I=[a,b]$ with $a<0<b$ such that  
for $t \in I$,  
$p+ t v$ has no zeros in $\D^2$.  
For our purposes, we can shrink the interval $I$ 
to be symmetric about $0$ and rescale $v$ so that
$p\pm v$ have no zeros in $\D^2$ and in fact $p\pm v$
are scattering stable.  
This is because the resultants associated to
$p,\tilde{p}$ will not be identically zero
and thus so will those of $p+tv$, $\tilde{p}+tv$
for $|t|$ small.

Write $v = \sum_{j=0}^{m} v_j(z) w^j$ and form
\[
V(z) = \begin{pmatrix} v_0(z) & v_1(z) & \dots & v_{m-1}(z) \\
 0 & v_0(z) & \dots & v_{m-2}(z)\\
 \vdots & 0 & \ddots & \vdots \\
 0 & 0 & 0 & v_{0}(z)
 \end{pmatrix}.
 \]

All of the matrix functions in the rest of this 
proof will be functions of $z$, so we will
omit the evaluations such as $V(z)$ and simply 
write $V$.
Since $p\pm v$ are scattering stable, 
 $(P\pm V)^*(P\pm V)-(Q\pm V)^*(Q\pm V) > 0$ for
all but finitely many $z\in \T$.

By the matrix Fej\'er-Riesz lemma, we can factor 
 \begin{equation} \label{factor}
 (P\pm V)^*(P\pm V)-(Q\pm V)^*(Q\pm V) = E_{\pm}^*E_{\pm}
 \end{equation} 
 where $E_{+}, E_{-}$ are matrix polynomials of degree at most $n$
with $\det E_{\pm}(z)$ non-zero for $z\in \D$.  
 We can also factor 
\begin{equation}\label{Epoly}
P^*P-Q^*Q = E^*E
\end{equation}
 where $E$ is a matrix polynomial of degree at most $n$ 
and is
non-singular whenever $z\in \D$. 

Let $r(z)$ be the resultant
of $p$ and $\tilde{p}$ with respect to $w$,
and let $r_{\pm}(z)$ be
the resultants of $p\pm v, \tilde{p}\pm v$ 
with respect to $w$ (for the choices of $+$ and $-$).
The number of roots of the
resultant $r(z)$ on $\T$
equals the number of common roots of $p$
and $\tilde{p}$ on $\T^2$ 
since $p$ has no zeros on $\D\times \T$
by Lemma \ref{scatnozonface}. 
More precisely, the multiplicity of a given
root of $r(z)$, say $z_0$, counts the number
of common roots of $p$ and $\tilde{p}$ on the line $z=z_0$.
Note that $r(z) = z^{mn} |\det(E)|^2$ and
$r_{\pm}(z) = z^{mn} |\det(E_{\pm})|^2$ on $\T$.

The left side of \eqref{factor} is
 \[
 P^*P-Q^*Q -(\pm 1)2\text{Re}((P-Q)^*V)
 \]
and averaging over $+$ and $-$ yields
\[
P^*P-Q^*Q = E^*E = \frac{1}{2}(E_{+}^*E_{+} + E_{-}^*E_{-})
\]
on $\T$.  Setting $\Phi_{\pm} = E_{\pm} E^{-1}$ we have 
$2I = \Phi_{+}^*\Phi_{+} + \Phi_{-}^{*}\Phi_{-}$ on $\T$ 
away from poles of $E^{-1}$.
But, this equation shows $\Phi_{\pm}$ are bounded
on $\T$ and so cannot have any poles.  So, $\Phi_{\pm}$
are analytic on $\cd$ since we already know
$E^{-1}$ has no poles in $\D$.

We conclude that on $\T$, $r_{\pm}(z) = r(z) |\det(\Phi_{\pm})|^2$.
Let $I_{\{z=z_0\}}(p,q)$ denote the number
of common zeros of two polynomials on the line $z=z_0$ counting
multiplicities.  We have just proved the following.

\begin{prop} \label{proplines}
Suppose $p\in \C[z,w]$ is scattering stable,
and suppose $v\in \C[z,w]$ satisfies $v(0)=0, \deg v= \deg \tilde{p}, v=\tilde{v}$,
as well as $p\pm v$ are scattering stable.
Then, for $z_0\in \T$
\[
I_{\{z=z_0\}}(p,\tilde{p}) \leq I_{\{z=z_0\}} (p\pm v, \tilde{p}\pm v).
\]
Moreover, both intersection multiplicities are even.
\end{prop}

The last statement follows from the formula $r(z) = z^{mn} |\det E |^2$
which implies that zeros on $\T$ occur with even multiplicity.

If the common zeros of $p$ and $\tilde{p}$ on $\T^2$ had distinct
$z$-coordinates (or $w$-coordinates) 
then we would be finished with Theorem \ref{thmmults}.  

At this stage the reader can jump to the next section to 
finish the proof of Theorem \ref{thmsaturated} or Section \ref{secmults}
to finish the proof of Theorem \ref{thmmults}.

\section{Completion of the proof of Theorem \ref{thmsaturated}}

Assume now that $p$ is $\T^2$-saturated and $p-\tilde{p}$ 
is irreducible.
Then, the resultant of $p$ and $\tilde{p}$, $r(z)$, 
has $2nm$ roots on $\T$ and these must be
all of its roots since $r(z)$
has degree at most $2nm$.  
This implies $\det(P^*P - Q^*Q) = \bar{z}^{nm}r(z)$ has $2nm$ roots on $\T$ 
and therefore $\det E$ has $nm$ roots on $\T$. 
 These must be all of the roots of $\det E$ 
since $E$ is $m\times m$ and of degree at most $n$.
  Thus, $E$ is invertible in $\D$ and $\C\setminus \cd$. 
Also, $z^nE(1/z)$ is an $m\times m$ matrix polynomial of degree at most $n$, 
which implies $\det (z^n E(1/z))$ has degree at most $nm$. 
 Since $\det E$ has $nm$ roots on $\T$, 
we can conclude that $\det (z^n E(1/z))$ has all of its roots on $\T$. 
Thus, $z^nE(1/z)$ is invertible at $0$.

At the same time, 
\[
E_{\pm}(1/z)E^{-1}(1/z)=(z^n E_{\pm}(1/z))(z^n E(1/z))^{-1}
\]
is analytic on $\D$ since $z^n E(1/z)$ is invertible in $\D$.  This implies $\Phi_{\pm}$ are analytic on the Riemann sphere, which implies $\Phi_{\pm}$ are constant $m\times m$ matrices.

Subtract equation \eqref{factor} with $+$ from $-$ to obtain
\begin{equation} \label{subtract}
4\text{Re}((P-Q)^*V) = E_{-}^*E_{-}-E_{+}^*E_{+}= E^*(\Phi_{-}^*\Phi_{-}-\Phi_{+}^*\Phi_{+}) E.
\end{equation}
Now we consider the matrices obtained from $p+t v$ where $t$ is a real parameter.  The matrices we get are
\[
(P+t V)^*(P+t V)-(Q+t V)^*(Q+t V)= E^*(I - t\frac{1}{2}\Phi_{-}^*\Phi_{-}+t\frac{1}{2}\Phi_{+}^*\Phi_{+})E.
\]
Unless $\Phi_{-}^*\Phi_{-} = \Phi_{+}^*\Phi_{+}$, there will be a value of $t$ such that 
\[
(I - t(\frac{1}{2}\Phi_{-}^*\Phi_{-}-\frac{1}{2}\Phi_{+}^*\Phi_{+}))
\] 
is singular since all matrices involved are self-adjoint.  
For such $t$, $\det((P+t V)^*(P+t V)-(Q+t V)^*(Q+t V))$ is identically zero on $\T$
and by Lemma \ref{commonfactor}  
this implies $p+tv$ and $\tilde{p}+tv$ have a common factor.
Such a common factor must be proper (meaning it has strictly
lower degree in one of the variables).  

Indeed, if it is
not proper, then $p+tv$ divides $\tilde{p}+tv$ or vice versa.
The latter possibility is excluded by the fact $p(0)\ne 0=\tilde{p}(0)=v(0)$.
In the former case, $\tilde{p}+tv = (p+tv) g$ 
for some non-constant polynomial $g$ with $g(0)=0$.
(Note we could not immediately rule this case out based on
degrees because $p+tv$ only has degree at most $(n,m)$.)
Then, $p-\tilde{p} = (p+tv)(1-g)$ is reducible contrary to our assumption.
However, a proper common factor of $p+tv$ and $\tilde{p}+tv$
implies a proper factor of $p-\tilde{p}$, which again contradicts
irreducibility of $p-\tilde{p}$.

We are left with the possibility $\Phi_{-}^*\Phi_{-}=\Phi_{+}^*\Phi_{+}$.
By \eqref{subtract} we have
$\text{Re}((P-Q)^*V) \equiv 0$, which implies 
\[
(P+t V)^*(P+t V)-(Q+t V)^*(Q+t V) = P^*P-Q^*Q
\] 
 for all $t\in \R$.  By \eqref{scformula}, this implies 
for $z\in \T$, $w,\eta\in \C$ that the following expression does
not depend on $t$
\[
\frac{\overline{(p+tv)(z,\eta)}(p+tv)(z,w) - \overline{(\tilde{p}+tv)(z,\eta)}(\tilde{p}+tv)(z,w)}{1-\bar{\eta}w}.
\]
The coefficient of $t$ is therefore identically
zero yielding
 \[
 \overline{(p-\tilde{p})(z,\eta)} v(z,w)+\overline{v(z,\eta)}(p-\tilde{p})(z,w)=0
 \]
 for $w,\eta\in \C, z\in \T$.  
Since $z\in \T$ and $v=\tilde{v}$,
if we replace $\eta$ with $1/\bar{\eta}$ and multiply through by $z^n\eta^m$ we get
 \[
 (\tilde{p}-p)(z,\eta) v(z,w)+ v(z,\eta)(p-\tilde{p})(z,w)=0
 \]
 for $w,\eta \in \C$.
As this is a polynomial in $z$ the identity holds for all $z\in \C$, not just $\T$.
 Since $p-\tilde{p}$ is assumed to be irreducible, either:
\begin{enumerate}
 \item $(p-\tilde{p})(z,w)$ divides 
 $(\tilde{p}-p)(z,\eta)$ or
\item  $(p-\tilde{p})(z,w)$ divides $v(z,w)$.  
\end{enumerate}

In addressing these two cases it is useful to point out that $p-\tilde{p}$ has
bidegree exactly $(n,m)$.  This is
because $p(0)\ne 0$ and the reflection of $p-\tilde{p}$ at the degree $(n,m)$
is $\tilde{p}-p$ which implies the coefficient of $z^nw^m$ in $(p-\tilde{p})(z,w)$ is
$-\overline{p(0)}$.

Case (1) implies $(p-\tilde{p})(z,w)$ does not
depend on $w$ which is only possible if $m=0$.
But,  $p-\tilde{p}$ being irreducible means $n=0$ or $n=1$ by
the fundamental theorem of algebra.  Having $n=0$ is not allowed and
when $n=1$ we automatically have $v=0$.  

Case (2) implies $v$ is a multiple of $p-\tilde{p}$ since $v$ has degree
at most $(n,m)$ and $p-\tilde{p}$ has degree exactly $(n,m)$. Since $p(0)\ne 0$
this is impossible unless $v=0$.

In all cases $v=0$.  This 
means the face of $f$ is the singleton $\{f\}$ and 
proves $f$ is an extreme point. 

\section{Completion of the proof of Theorem \ref{thmmults}}\label{secmults}
We pick up where we left off in Section \ref{prelims}.
Proposition \ref{proplines} shows
the number of common zeros of $p$ and $\tilde{p}$
on $\T^2$
is less than or equal to the number of 
common zeros of $p\pm v$ and $\tilde{p}\pm v$ on $\T^2$.
However, since the resultant is a global object 
we have only proven that the number of zeros on ``vertical'' lines 
in $\T^2$ increases.  
To fix this we cannot do a linear change of variables
as is often done when using the resultant to count 
common zeros (as in \cite{Fischer} for instance).
Instead we perform a more complicated change of variables
and use a different (but equivalent) definition of
intersection multiplicity.  

First, recall that the intersection
multiplicity of $p,q$ at a common zero $x\in \C^2$ 
is equal to 
\[
\dim \mathcal{O}_x/(p,q)\mathcal{O}_x.
\]
Here $\mathcal{O}_x$ is the localization
of $\C[z,w]$ at $x$, namely the ring of rational functions
with denominators non-vanishing at $x$.
Also, $(p,q)\mathcal{O}_x$ denotes the ideal generated by
$p,q$ within the local ring $\mathcal{O}_x$.  
This definition of intersection multiplicity is used
in \cite{Fulton}.

Now, suppose $p, \tilde{p}$ have a common zero on $\T^2$
which we assume without loss of generality is $\mathbf{1}=(1,1)$.
The map 
\[
\phi_k: F(z,w) \mapsto F(w^k z,w)
\]
is an isomorphism of the local ring $\mathcal{O}_{\mathbf{1}}$
onto itself
since it has inverse $G(z,w)\mapsto G(w^{-k} z, w)$.
This implies that
\begin{equation} \label{phik}
I_{\mathbf{1}}(p,\tilde{p}) = I_{\mathbf{1}}(\phi_k(p), \phi_k(\tilde{p})).
\end{equation}
This applies to $p\pm v$ as well.

Since $p$ is scattering stable, so is $\phi_k(p)$.  Indeed,
$\phi_k(p)$ has no zeros in $\D^2$ since $(z,w) \mapsto (w^k z, w)$
maps the bidisk into itself.  
Also, 
\[
\widetilde{\phi_k(p)} = z^{n} w^{kn+m} \overline{p(1/\bar{w}^k1/\bar{z}, 1/\bar{w})} = \phi_k(\tilde{p})
\]
If $\phi_k(p)$ and $\phi_k(\tilde{p})$ had
a nontrivial common factor say $\phi_k(p)=g h_1$, $\phi_k(\tilde{p}) = g h_2$ then
\[
p(z,w) = g(w^{-k} z,w) h_1(w^{-k} z,w)
\qquad
\tilde{p}(z,w)= g(w^{-k} z, w) h_2(w^{-k} z,w).
\]
One can then argue that $g(w^{-k} z,w)$ is either already a polynomial
or there exists $N$ such that $w^N g(w^{-k} z,w)$, $w^{-N}h_1(w^{-k} z,w)$, $w^{-N}h_2(w^{-k} z,w)$
are polynomials.  Either case implies $p$ and $\tilde{p}$ have a non-trivial common 
factor contrary to assumption.

The same arguments apply to $p\pm v$.  Therefore,
by Proposition \ref{proplines}
\[
I_{\{z=1\}}(\phi_k(p), \phi_k(\tilde{p})) \leq I_{\{z=1\}}(\phi_k(p)\pm \phi_k(v), \phi_k(\tilde{p})\pm \phi_k(v)).
\]
Finally, we can choose $k$ so that $\phi_k(p), \phi_k(\tilde{p})$ 
and $\phi_k(p)\pm \phi_k(v), \phi_k(\tilde{p})\pm \phi_k(v)$ have a single
common zero on the line $\{z=1\}$ namely $\mathbf{1}=(1,1)$. 
Then, the above intersection multiplicities
count the multiplicity at just $\mathbf{1}$ and these agree 
with the multiplicities associated to $p,\tilde{p}$ and $p\pm v, \tilde{p}\pm v$
as mentioned in equation \eqref{phik}.

This proves 
\[
I_{\mathbf{1}}(p,\tilde{p}) \leq I_{\mathbf{1}} (p\pm v, \tilde{p} \pm v)
\]
which proves Theorem \ref{thmmults}.

A corollary of the proof is the following fundamental fact.

\begin{corollary}
If $p\in \C[z,w]$ is scattering stable then for $\zeta \in \T^2$, 
$I_{\zeta}(p,\tilde{p})$
is even.
\end{corollary}

This follows because the intersection multiplicities in Proposition \ref{proplines}
are even.  We gave a detailed proof of this corollary in Appendix C of \cite{GK}.

\section{Degree $(1,1)$ polynomials} \label{modest}

\begin{proof}[Proof of Theorem \ref{oneone}]
Suppose $p$ is scattering stable, has degree $(1,1)$,
and $f=(p+\tilde{p})/(p-\tilde{p})$ is
an extreme point.

We can normalize so that $|p(0)|=1$. 
(We cannot normalize so that $p(0)=1$
because $p/p(0)$ has reflection $\tilde{p}/\overline{p(0)}$
not $\tilde{p}/p(0)$.)

Since $\tilde{p}(0)=0$, 
$p$ has the form $p(z,w) = \mu+az+bw$; $\mu\in \T, a,b,\in \C$.  
We can 
apply a rotation in each of the variables and
replace $p$ with the polynomial $\mu(1-|a|z-|b|w)$; 
we will drop the
absolute values on $a,b$ and assume $a,b >0$.  
This change of variables does not change extremality.

For $p$ to be stable we must have $1\geq a+b$.
Since $f$ is extreme, $p$ must have a zero on $\T^2$.
This forces $1=a+b$ and the zero is $(z,w)=(1,1)$.
Zeros on $\T^2$ occur with even multiplicity
and this automatically implies $p$ is $\T^2$-saturated.
(Thus, in this low degree situation a single zero forces
the polynomial to be saturated.)

We need to show $p-\tilde{p}$ is irreducible.  We have
\[
p-\tilde{p} = \mu+(\bar{\mu} b-\mu a)z+(\bar{\mu}a - \mu b)w - \bar{\mu}zw.
\]
If this were reducible it would factor into 
$\mu(1-cz)(1-dw)$.  We will omit the
details but the only way this can happen is if $c=d=1$ and 
$\mu = \pm i$.  We can normalize so $\mu = i$.  Then,
\[
p-\tilde{p} = i(1 - (z+w) +zw)
\]
while
\[
p+\tilde{p} = i(1+t(z-w) - zw)
\]
where $t=b-a$.
Then, one can directly check
\[
f = b \frac{1+z}{1-z} +a \frac{1+w}{1-w}
\]
contradicting extremality of $f$.
\end{proof}

The next theorem is an amusing consequence of the
above proof.

\begin{theorem}\label{thm:mobius}
Let $p\in \C[z_1,z_2]$ be scattering stable, $\tilde{p}(0,0)=0$, and $\deg p = (1,1)$.
Set $F = \tilde{p}/p$ and for $\nu \in \T$, 
\[
f_{\nu} = \frac{1+\nu^2 F}{1-\nu^2 F} = \frac{\bar{\nu} p + \nu \tilde{p}}{\bar{\nu} p - \nu \tilde{p}}.
\]
Then, there exists $\nu \in \T$ such that $f_{\nu}$ is not an extreme point of $PRP_2$.
\end{theorem}

I. Klep and J. Pascoe have found an example
which shows this theorem does not hold more generally.

\begin{proof}
Let $p\in \C[z,w]$ be scattering stable, $\deg p = (1,1)$, and $\tilde{p}(0,0)=0$.
We claim some $f_{\nu}$ as in the theorem statement is not extreme.
By Theorem \ref{oneone} if $p$ is not $\T^2$-saturated, then $f_{\mu}$ is
never extreme.  So, we assume $p$ is $\T^2$-saturated.  

If we follow the proof of Theorem \ref{oneone} above, we
see that if we set $\nu = i\mu$ then $\bar{\nu} p - \nu \tilde{p} = i(1-z)(1-w)$
which means $f_{\nu}$ is not extreme.  
\end{proof}
 
\section{Example of a non-saturated polynomial}\label{nonsaturated}

Let $p(z,w) = 3-z-z^2-w$ which is scattering stable, has degree $(2,1)$, 
and has a single zero on $\T^2$ at $(z,w)=(1,1)$.  
The multiplicity of this zero is $2$, so $p$ is not saturated (which
would require $4$ boundary zeros counting multiplicity).
One can compute that this multiplicity is $2$ using the resultant (with
respect to $w$) of
$p$ and $\tilde{p} = -z^2 -w - zw +3z^2w$ which is
\[
\det\begin{pmatrix} 3-z-z^2 & -1 \\ -z^2 & -1-z+3z^2 \end{pmatrix} = -(z-1)^2(3z^2+8z+3).
\]

Next, we claim that $f=\frac{p+\tilde{p}}{p-\tilde{p}}$ is not extreme.
To do this we must perturb $p$ with $v$ satisfying $v=\tilde{v}$, $v(0,0)=0$,
and stay stable.  Let $v(z,w) = (1-z)(z-w)$.  We claim $p+tv$ has no
zeros in $\D^2$ for $t\in [-7/4,1/2]$.  By the Schur-Cohn test
we check when the following expression is non-negative
on $\T$
\[
|3-z-z^2+t(1-z)z|^2 - |-z^2 + t(1-z)z|^2.
\]
This can be rewritten as 
\[
|1-z|^2(2t+8 + 6(t+1)\text{Re} z)
\]
which is non-negative on $\T$ for $t\in [-7/4,1/2]$. 
Thus, for such $t$, $p+tv$ has no zeros in $\T\times \cd$
except for finitely many points on $\T^2$ and finitely many
vertical lines by Lemma \ref{possc}.
We can rule out vertical lines since
\[
p-\tilde{p} = 3-z + (3z-z^2)w
\]
is irreducible (and if $p+tv$ had a factor $z-z_0$ with $z_0\in \T$
then so would $\tilde{p}+tv$ and so would $p-\tilde{p}$).
By Lemma \ref{scatnozonface} $p+tv$ has no zeros in $\T\times \D$.  
For $w\in \D$, the number of zeros 
of $z\mapsto (p+tv)(z,w)$ in $\D$ depends continuously 
on $w$ by the argument principle and hence must be constant.  
We can show $(p+tv)(z,0) = 3+(1-t)z-(1+t)z^2$ has no zeros
in $\D$ directly by noting $|t-1|+|t+1|\leq 3$ for $t\in [-7/4,1/2]$.
Therefore, $p+tv$ has no zeros in $\D^2$ for $t\in [-7/4,1/2]$.  
This proves $f$ is not extreme.  

The endpoints of the line segment
$q_{-}=p-7v/4$ and $q_+=p+v/2$ are
\[
q_{-} = 3-(11/4)z+(3/4)w -(3/4)z^2 -(7/4)zw
\]
\[
q_+ = 3-(1/2)z - (3/2)w-(3/2)z^2+(1/2)zw.
\]
Now, $q_{-}$ has two zeros on $\T^2$; $(1,1),(-1,1)$.  Each
has multiplicity 2 which implies $q_{-}$ is $\T^2$-saturated. 
Since $p-\tilde{p}= q_{-}-\tilde{q_{-}}$ is irreducible we conclude
that $f_{-} = \frac{q_{-} + \tilde{q}_{-}}{q_{-} - \tilde{q}_-}$ 
is an extreme point by Theorem \ref{thmsaturated}.

Next, $q_{+}$ has a single zero on $\T^2$; $(1,1)$.  It occurs
with multiplicity $4$ which implies $q_{+}$ is $\T^2$-saturated.
As with $q_{-}$, we see that the analogue of $f_{-}$, namely $f_+$,
is also an extreme point.

Then, 
\[
f = (2/9) f_{-} + (7/9) f_{+}
\]
expresses $f$ as a convex combination of extreme points.

\section{Integrability of perturbations}  

In this section we prove Theorem \ref{intthm}.
We need to use several definitions and results from \cite{GK}.

\begin{definition}
If $p\in \C[z,w]$ is scattering stable and $\vec{A}_1\in\C^{N}[z,w],\vec{A}_2\in \C^M[z,w]$
are vector polynomials satisfying
\[
|p|^2-|\tilde{p}|^2 = (1-|z|^2)|\vec{A}_1|^2 + (1-|w|^2)|\vec{A}_2|^2
\]
then we call $(\vec{A}_1,\vec{A}_2)$ an \emph{Agler pair}.
\end{definition}

For this definition, $N,M$ need not have any relation
to the bidegree $(n,m)$ of $p$.  One can in fact show $N\geq n, M\geq m$,
but this is not important for now.  Later we do look at Agler pairs with minimal
dimensions.
 
It is proven in \cite{CW, GKpnoz} that $p$ has at least
one Agler pair. (These vector polynomials are
closely related to so-called Agler kernels for
more general bounded analytic functions.  See \cite{BK}.)

Lemma 7.3 of \cite{GK} proves that the values of $(|\vec{A}_1|, |\vec{A}_2|)$
on $\T^2$ are the same for all Agler pairs.  
Therefore, the following proposition does
not depend on the particular Agler pair.

\begin{prop}\label{lprop}
Suppose $p\in \C[z,w]$ is scattering stable 
and $(\vec{A}_1,\vec{A}_2)$ is an Agler pair for $p$.

Given $q \in \C[z,w]$, $q/p \in L^2(\T^2)$
if and only if there is a constant $C$ such that
\[
|q| \leq C(|\vec{A}_1| + |\vec{A}_2|)
\]
holds on $\T^2$.  
\end{prop}
\begin{proof}
Lemma 7.3 of \cite{GK} proves explicit formulas
for $|\vec{A}_1|, |\vec{A}_2|$ on $\T^2$ in terms of $p$.
Corollary 7.4 of \cite{GK} proves that the given inequality
is equivalent to $q/p\in L^2(\T^2)$ however it is stated using
said explicit formulas from Lemma 7.3 (so that the corollary 
does not need to reference Agler pairs). 
\end{proof}

The first part of Theorem \ref{intthm} is contained in the following
proposition.

\begin{prop}\label{propvp}
Suppose $p$ is scattering stable, $v=\tilde{v}$, and $p\pm v$ are scattering stable.  
Let $q \in \C[z,w]$. If $\frac{q}{p+ v}\in L^2(\T^2)$ then $\frac{q}{p} \in L^2(\T^2)$.
In particular, $v/p \in L^2(\T^2)$. 
\end{prop}

\begin{proof}
We can rescale $v$ so that $p\pm v$ are scattering stable.
Let $(E_{\pm}, F_{\pm})$ be Agler pairs for $p\pm v$.
Since
\[
\begin{aligned}
|p\pm v|^2 - |\tilde{p} \pm v|^2 &= |p|^2-|\tilde{p}|^2 \pm 2\text{Re}(\overline{p-\tilde{p}}v)\\
&=(1-|z|^2)|E_{\pm}|^2 + (1-|w|^2)|F_{\pm}|^2
\end{aligned}
\]
it follows that for $E= \frac{1}{\sqrt{2}} \begin{pmatrix}  E_{+}\\ E_{-}\end{pmatrix}$ 
and $F=\begin{pmatrix} F_{+} \\ F_{-} \end{pmatrix}$, the pair $(E,F)$ is an Agler pair for $p$.
This comes from averaging the above formula over $+$ and $-$. 

Note that $|E| \geq \frac{1}{\sqrt{2}} |E_{\pm}|$, $|F| \geq \frac{1}{\sqrt{2}}|F_{\pm}|$
on $\T^2$.  These bounds along with Prop \ref{lprop} automatically
give that $q/(p+v) \in L^2(\T^2)$ implies $q/p \in L^2(\T^2)$ for $q\in \C[z,w]$.

Since $1 = (p+v)/(p+v) \in L^2(\T^2)$ we conclude that $(p+v)/p \in L^2$ and hence
$v/p\in L^2$.
\end{proof}
 
For any scattering stable polynomial $p\in \C[z,w]$ with
degree at most $(n,m)$ define the subspace
\[
K_p = \{q\in\C[z,w]: \deg q\leq (n-1,m-1) \text{ and } q/p\in L^2(\T^2)\}.
\]
Theorem B of \cite{GK} states that 
\[
\dim K_p = nm- \frac{1}{2} I_{\T^2}(p,\tilde{p})
\]
where $I_{\T^2}(p,\tilde{p})$ is the number
of common zeros of $p$ and $\tilde{p}$ on $\T^2$ counted with
multiplicities.  We can then conclude the following.

\begin{corollary}
Under the assumptions of Proposition \ref{propvp} we have
\[
\dim K_{p} \geq \dim K_{p+v}.
\]
\end{corollary}

This also follows from Theorem \ref{thmmults}.  
As mentioned in the introduction $p$ is
$\T^2$-saturated if and only if $K_p = \{0\}$ (see Corollary 6.5 of \cite{GK}).  
We use this
fact in the proof of the second half of Theorem \ref{intthm}
which is stated in the following proposition.

\begin{prop}
Suppose $p \in \C[z,w]$ is scattering stable and
not $\T^2$-saturated.  Then, there exists nonzero $v\in \C[z,w]$
with $\deg v \leq (n,m), v=\tilde{v}, v(0,0)=0$ and $v/p \in L^2(\T^2)$.
\end{prop}

\begin{proof}
Since $p$ is not $\T^2$-saturated, there
exists $f \in K_p$, $f\ne 0$.  We can 
assume $z$ does not divide $f$ because otherwise $f/z \in K_p$ 
and we could divide out all of the factors of $z$ if necessary.

Let $\tilde{f}$ be the reflection of $f$ at degree $(n-1,m-1)$.
Set $v = w f + z \tilde{f}$.  Note $v$ is nonzero because $z$ does not
divide $f$.  Since $|f|=|\tilde{f}|$ on $\T^2$, we have $\tilde{f}/p \in L^2(\T^2)$.
Therefore, $v/p \in L^2(\T^2)$.  The other properties of $v$ are straightforward.
\end{proof}

\section{Determinantal representations and transfer function realizations}\label{secdets}
In this section we discuss a number of formulas which
are known to hold for scattering stable $p$ and corresponding $F=\tilde{p}/p$ 
in the case of two variables.  When $p$ is saturated it turns out
the formulas have a special form.  The techniques of this section are well-known
and this section is more of a supplement to the paper.  See \cite{Drexel, GKdv} 
for related results.

It is a well-established consequence 
of the sums of squares formula\eqref{sos} that if $p \in \C[z_1,z_2]$ is
scattering stable of bidegree $n=(n_1,n_2)$ then 
there exists a $(1+|n|)\times (1+|n|)$ unitary $U = \begin{pmatrix} A & B \\ C& D\end{pmatrix}$ 
where $A$ is $1\times 1$ such that
\begin{equation} \label{tfr}
F(z) = A+ B P(z) (1-D P(z))^{-1} C.
\end{equation}
Here 
\begin{equation} \label{Pmatrix}
P(z) = \begin{pmatrix} z_1 I_{n_1} & 0 \\ 0 & z_2 I_{n_2} \end{pmatrix}.
\end{equation}

We call the formula \eqref{tfr} a \emph{unitary transfer function realization} of $F$.

Since $F$ has denominator $p$, we see that $p$ divides $\det(I-D P(z))$.
But, $\det(I-D P(z))$ has bidegree at most $n$ 
we must have $p(z) = p(0) \det(I-D P(z))$.
Thus, $p$ has a \emph{contractive determinantal representation};
so-called because $D$ is a contractive matrix.
The matrix $D$ is not just contractive; 
it is also a rank one perturbation of a unitary matrix.
Indeed, the matrices 
\[
V_{\alpha} = D+\frac{\alpha}{1-\alpha A} CB
\]
are unitary for $\alpha \in \T$.
This follows from 
\[
U \begin{pmatrix} \frac{\alpha}{1-\alpha A} B \\ I_{|n|} \end{pmatrix} =
 \begin{pmatrix} \frac{1}{1-\alpha A} B \\ V_{\alpha} \end{pmatrix}
\]
which implies
\[
I-V_{\alpha}^*V_{\alpha} = \frac{1-|\alpha|^2}{|1-\alpha A|^2}B^*B
\]
and this is zero when $|\alpha|=1$.

Notice now that the numerator of $F(z)$ is 
\[
\tilde{p}(z) = p(0)\det(I-DP(z))F(z) =p(0)( A\det(I-DP(z)) +BP(z) \text{adj}(1-DP(z)) C)
\]
where $\text{adj}$ denotes the classical adjoint or adjugate matrix.
Observe that
\[
\begin{aligned}
p-\alpha \tilde{p} &= p(0)((1-\alpha A)\det(I-DP(z)) - \alpha BP(z) \text{adj}(I-DP(z))C)\\
&= p(0)(1-\alpha A)\det( I-DP(z) -\frac{\alpha}{1-\alpha A} CB P(z))\\
&= p(0)(1-\alpha A) \det( I - V_{\alpha} P(z)).
\end{aligned}
\]
The second line used the ``rank one matrix update formula''
\[
\det(M + xy^*) = \det M + y^*\text{adj}(M) x.
\]
We conclude that $p-\alpha \tilde{p}$ has a \emph{unitary determinantal representation}
whenever $\alpha \in \T$.  This gives a general way to describe the denominators
of elements of $PRPR_2$.

The matrix $U$ can be chosen with special structure when $p$ is saturated.

\begin{theorem}
If $p\in \C[z_1,z_2]$ is $\T^2$-saturated, then $F=\frac{\tilde{p}}{p}$ possesses
a symmetric unitary transfer function realization; i.e.
$U$ in \eqref{tfr} can be chosen to be a unitary and symmetric: $U^*=U^{-1}$ and $U=U^t$.
In particular, $p$ has a determinantal representation
using a symmetric contraction (which is a rank one perturbation of a symmetric unitary)
and $p-\tilde{p}$ has a symmetric unitary determinantal representation.
\end{theorem}

This theorem can be subdivided
into two parts.

\begin{prop}
If $p\in \C[z_1,z_2]$ is $\T^2$-saturated with $\deg p = (n_1,n_2)$
then 
there exist $E_1 \in \C^{n_1}[z_1,z_2], E_2 \in \C^{n_2}[z_1,z_2]$
of bidegrees at most $(n_1-1,n_2), (n_1,n_2-1)$
such that
\begin{equation} \label{sos2}
|p(z)|^2 -|\tilde{p}(z)|^2 = (1-|z_1|^2)|E_1(z)|^2 + (1-|z_2|^2)|E_2(z)|^2
\end{equation}
and
\begin{equation} \label{symmetry}
E_1(z) = \underset{:=\tilde{E}_1}{\underbrace{z_1^{n_1-1} z_2^{n_2} \overline{E_1(1/\bar{z}_1, 1/\bar{z}_2)}}}
\qquad E_2(z) = \underset{:=\tilde{E}_2}{\underbrace{z_1^{n_1} z_2^{n_2-1} \overline{E_2(1/\bar{z}_1,1/\bar{z}_2)}}}.
\end{equation}
\end{prop}

This follows directly from 
Theorem 1.1.5 of \cite{GKpnoz} and Corollary 13.6 of \cite{GK}.
The condition in \eqref{symmetry} can be replaced with
the weaker conditions $|E_1|=|\tilde{E}_1|$, $|E_2| = |\tilde{E}_2|$.
(An argument is given in 
the proof of Theorem 1.1.5 of \cite{GKpnoz}.  
It uses the fact that any symmetric unitary $U$ can be
factored as $U=V^tV$ where $V$ is unitary.)

\begin{prop} \label{symmetryprop}
If $p$ is scattering stable and possesses a sums of squares decomposition
as in \eqref{sos2} which in addition satisfies
\eqref{symmetry}, then $F=\frac{\tilde{p}}{p}$ satisfies \eqref{tfr}
where $U$ may be chosen to be symmetric.
\end{prop}

Symmetry in the
sums of squares formulas does not   
characterize saturated polynomials.

\begin{example}
Let $p = (3-z_1-z_2)^2$---a non-saturated polynomial.  Then, $\tilde{p} = (3z_1 z_2-z_1-z_2)^2$
and \eqref{sos2} holds with
\[
E_1 = \begin{pmatrix} q(z_2)p(z) \\ \tilde{q}(z_2) \tilde{p}(z) \end{pmatrix} 
\qquad E_2 = \begin{pmatrix}  \tilde{q}(z_1) p(z) \\ q(z_1) \tilde{p}(z) \end{pmatrix}
\]
where $q(\zeta) = \sqrt{3}\left( \frac{1-\sqrt{5}}{2} + \frac{1+\sqrt{5}}{2} \zeta\right)$
and $\tilde{q}$ is the reflection at degree $1$ of this one variable polynomial.
Evidently, $|E_j|=|\tilde{E}_j|$.
\end{example}

It may be possible to use
the characterization of all sums of squares
decompositions \eqref{sos2} given in \cite{GK} 
to characterize which $p$ have decompositions satisfying
\eqref{symmetry}.

We now prove Proposition \ref{symmetryprop}.
Unfortunately, the only way we see to 
give a comprehensible proof is to rehash 
a number of standard arguments (including 
the proof of \eqref{tfr}).  

\begin{proof}[Proof of Proposition \ref{symmetryprop}]
Initially we do not need \eqref{symmetry}.

Equation \eqref{sos2} can be polarized to the following
form
\begin{equation} \label{polarized}
\overline{p(w)}p(z) - \overline{\tilde{p}(w)}\tilde{p}(z) = 
(1-\bar{w}_1 z_1) E_1(w)^* E_1(z) + (1-\bar{w}_2 z_2) E_2(w)^* E_2(z)
\end{equation}
by the polarization theorem for holomorphic functions.  
We now use what is called a lurking isometry argument.
Let $N=1+n_1+n_2$.
Define 
\[
X(z):= \begin{pmatrix} p(z) \\ z_1 E_1(z) \\ z_2 E_2(z) \end{pmatrix} \qquad 
Y(z) := \begin{pmatrix} \tilde{p}(z) \\ E_1(z) \\ E_2(z) \end{pmatrix}.
\]
Then, \eqref{polarized} rearranges into $X(w)^*X(z) = Y(w)^*Y(z)$.
The map
\begin{equation} \label{mapsto}
X(z) \mapsto 
Y(z) 
\end{equation}
extends to a well-defined linear isometry from the span of the elements on 
the left to the span of the elements on the right (as $z$ varies over $\C^2$).
It turns out that $\text{span}\{X(z): z\in \C^2\} = \C^{N}$ and this
implies the map above actually extends uniquely to a unitary.
This follows from the construction of $E_1,E_2$ in \cite{GKpnoz}.  
However, as this fact is difficult to retrieve from \cite{GKpnoz} without
introducing too much machinery we proceed without it.  
The map \eqref{mapsto} can always be extended to some unitary; we would like
to show that if $E_1,E_2$ satisfy \eqref{symmetry} then we can extend to 
a symmetric unitary.  

So, suppose \eqref{symmetry} holds and 
notice that $Y = \tilde{X}$, with the reflection performed at the
degree $(n_1,n_2)$.
We introduce the orthogonal
complements of the spans of the left and right sides in \eqref{mapsto}
\[
\begin{aligned}
S_1 &=\text{span}\{X(z):z\in \C^2\}^{\perp} \\
S_2 &=\text{span}\{\tilde{X}(z): z \in \C^2\}^{\perp} 
\end{aligned}
\]
Notice that $v \in S_1$ iff $\overline{v} \in S_2$ because we can reflect
the equation $X^* v= 0$ to obtain $\tilde{X}^* \overline{v} = 0$.
Let $v_1,\dots, v_k$ be an orthonormal basis for $S_1$.  Then,
$\bar{v}_1,\dots, \bar{v}_k$ is an orthonormal basis for $S_2$.
We can then extend the map \eqref{mapsto} to all of $\C^N$
by mapping $v_j \mapsto \bar{v}_j$ 
and extending linearly.  Note that we are not saying $v\mapsto \bar{v}$ for
a general $v\in S_1$.  We get an $N\times N$ unitary matrix $U$ with the property
\[
U X(z) = \tilde{X}(z) \text{ and } U v_j = \bar{v}_j.
\]
Applying the reflection operation at the degree $(n_1,n_2)$ and conjugation reveals
\[
U^t X(z) = \tilde{X}(z) \text{ and } U^t v_j = \bar{v}_j.
\]
Therefore, $U=U^t$.

Write $U = \begin{pmatrix} A & B \\ C & D \end{pmatrix}$ where $A$ is $1\times 1$;
the sizes of $B,C,D$ are then determined.
Letting $E = \begin{pmatrix} E_1 \\ E_2 \end{pmatrix}$
we have
\[
\begin{aligned}
A p(z) + B P(z) E(z) &= \tilde{p} \\
C p(z) + D P(z) E(z) &= E(z)
\end{aligned}
\]
where $P(z)$ is given in \eqref{Pmatrix}.
The second line implies $p(z)(I-D P(z))^{-1}C = E(z)$ which implies 
$A+BP(z)(I-DP(z))^{-1}C = \tilde{p}/p$ via the first line.
\end{proof}

\section{Acknowledgments}
I would like to thank Mike Jury, John McCarthy, and James Pascoe for useful 
discussions.  Thanks to Pascoe and Klep for disproving 
a conjectured generalization of Theorem \ref{thm:mobius}.

\end{document}